\documentclass[12pt]{amsart}
\usepackage{amssymb,verbatim,enumerate,ifthen}
\usepackage{cite}
\usepackage[mathscr]{eucal}
\usepackage[utf8]{inputenc}
\usepackage[T1]{fontenc}
\usepackage{esint}
\usepackage{marginnote}
\newtheorem{thm}{Theorem}[section]
\newtheorem{cor}[thm]{Corollary}
\newtheorem{prop}[thm]{Proposition}
\newtheorem{lem}[thm]{Lemma}

\theoremstyle{definition}
\newtheorem{defin}[thm]{Definition}

\numberwithin{equation}{section}
\def\eq#1{{\rm(\ref{#1})}}
\def\Eq#1#2{\ifthenelse{\equal{#1}{*}}
  {\begin{equation*}\begin{aligned}[]#2\end{aligned}\end{equation*}}
  {\begin{equation}\begin{aligned}[]\label{#1}#2\end{aligned}\end{equation}}}

\def\A{\mathscr{A}}

\def\M{\mathscr{M}}
\def\P{\mathscr{P}}
\newcommand{\operator}[1]{\mathop{\vphantom{\sum}\mathchoice
{\vcenter{\hbox{\LARGE $#1$}}}
{\vcenter{\hbox{\Large $#1$}}}{#1}{#1}}\displaylimits}

\def\Mm{\operator{\mathscr{M}}}
\def\Ar{\operator{\mathscr{A}}}

\newcommand\R{\mathbb{R}}
\newcommand\N{\mathbb{N}}
\newcommand\Z{\mathbb{Z}}

\newcommand\Q{\mathbb{Q}}

\def\WQX{V(\Q)}
\def\WRX{V(\R)}
\newcommand{\QA}[1]{\A_{#1}}

\newcommand{\vone}{\textbf1}

\newcommand{\dotvec}[3][SKIPPED]{
\ifthenelse{\equal{#1}{SKIPPED}}
  {#2,\dots,#3}
  {\underbrace{#2,\dots,#3}_{#1\text{ entries}}}
}

\newcommand{\Hc}[2][SKIPPED]{
\ifthenelse{\equal{#1}{SKIPPED}}
  {
    \ifthenelse{\equal{#2}{}}
      {\mathscr{H}}
      {\mathscr{H}(#2)}
  }
  {
    \ifthenelse{\equal{#2}{}}
      {\mathscr{H}_{#1}}
      {\mathscr{H}_{#1}(#2)}
  }
}
\newcommand{\Est}[2][SKIPPED]{
\ifthenelse{\equal{#1}{SKIPPED}}
  {
    \ifthenelse{\equal{#2}{}}
      {\mathscr{C}}
      {\mathscr{C}(#2)}
  }
  {
    \ifthenelse{\equal{#2}{}}
      {\mathscr{C}_{#1}}
      {\mathscr{C}_{#1}(#2)}
  }
}
\title
{On Hardy type inequalities for weighted means}

\author{Zsolt P\'ales}
\thanks{[Zsolt P\'ales] ORCID: 0000-0003-2382-6035. Corresponding author. Institute of Mathematics, University of Debrecen, Pf.\ 12, 4010 Debrecen, Hungary, e-mail: pales@science.unideb.hu. 
The research was supported by the Hungarian Scientific Research Fund (OTKA) Grant K-111651 and by the EFOP-3.6.1-16-2016-00022 project. This project is co-financed by the European Union and the European Social Fund.}

\author{Pawe{\l} Pasteczka}
\thanks{[Pawe{\l} Pasteczka] ORCID: 0000-0002-8593-0025. Institute of Mathematics, Pedagogical University of Cracow,  Podchor\k{a}\.{z}ych str 2, 30-084 Cracow, Poland, e-mail: pawel.pasteczka@up.krakow.pl}

\begin{document}
\begin{abstract}
The aim of this paper is to establish weighted Hardy type inequality in a broad family of means. In other words, for a fixed vector of weights $(\lambda_n)_{n=1}^\infty$ and a weighted mean $\mathscr{M}$, we search for the smallest number $C$ such that
$$\sum_{n=1}^{\infty} \lambda_n \mathscr{M} \big((x_1,\dots,x_n),(\lambda_1,\dots,\lambda_n)\big) \le C \sum_{n=1}^{\infty} \lambda_nx_n \text{ for all admissible }x.$$

The main results provide a definite answer in the case when $\mathcal{M}$ is monotone and satisfies the weighted counterpart of the Kedlaya inequality. In particular, if $\mathcal{M}$ is symmetric, Jensen-concave, and the sequence $\big(\tfrac{\lambda_n}{\lambda_1+\cdots+\lambda_n}\big)$ is nonincreasing. In addition, it is proved that if $\mathcal{M}$ is a symmetric and monotone mean, then the biggest possible weighted Hardy constant is achieved if $\lambda$ is the constant vector.
\end{abstract}


\maketitle

\section{Introduction}
In twenties of the last century several authors, motivated by a conjecture by Hilbert, proved that 
\Eq{H}{
  \sum_{n=1}^\infty \P_p(x_1,\dots,x_n) \le C(p) \sum_{n=1}^\infty x_n,
}
for every sequences $(x_n)_{n=1}^\infty$ with positive terms, where $\P_p$ denotes the $p$-th \emph{power mean} 
(extended to the limiting cases $p=\pm\infty$) and
\Eq{*}{
C(p):=
\begin{cases} 
1 & p=-\infty, \\
(1-p)^{-1/p}&p \in (-\infty,0) \cup (0,1), \\ 
e & p=0, \\
\infty & p\in[1,\infty],
\end{cases} 
}
and this constant is sharp, i.e., it cannot be diminished. 

First result of this type with nonoptimal constant was established by Hardy in the seminal paper \cite{Har20a}. Later it was improved and extended by Landau \cite{Lan21}, Knopp \cite{Kno28}, and Carleman \cite{Car32} whose results are summarized in inequality \eq{H}. Meanwhile Copson \cite{Cop27} adopted Elliott's \cite{Ell26} proof of the Hardy inequality to show (in an equivalent form) that if $\P_p(x,\lambda)$ denotes the $p$-th \emph{$\lambda$-weighted power mean of a vector $x$}, then
\Eq{E:EllCop}{
\sum_{n=1}^{\infty} \lambda_n \P_p\big((x_1,\dots,x_n),(\lambda_1,\dots,\lambda_n)\big) \le C(p) \sum_{n=1}^{\infty} \lambda_nx_n
}
for all $p\in (0,1)$, and sequences $(x_n)_{n=1}^\infty$ and $(\lambda_n)_{n=1}^\infty$ with positive terms. More about the history of the developments related to Hardy type inequalities is sketched in catching surveys by Pe\v{c}ari\'c--Stolarsky \cite{PecSto01}, Duncan--McGregor \cite{DucMcG03}, and in a recent book of Kufner--Maligranda--Persson \cite{KufMalPer07}.

In a more general setting, for a given mean $\M \colon \bigcup_{n=1}^{\infty} I^n \to I$ (where $I$ is a real interval 
with $\inf I=0$), let $\Hc\M$ denote the smallest nonnegative extended real number, called the \emph{Hardy 
constant of $\M$}, such that
\Eq{*}{
\sum_{n=1}^\infty \M(x_1,\dots,x_n)\le \Hc\M\sum_{n=1}^\infty x_n
}
for all sequences $(x_n)_{n=1}^\infty$ belonging to $I$. If $\Hc\M$ is finite, then we say that $\M$ is 
a \emph{Hardy mean}. In this setup, a $p$-th power mean is a Hardy mean if and only if $p\in[-\infty,1)$ and 
$\Hc{\P_p}=C(p)$ for all $p\in[-\infty,\infty]$.

 For investigating the Hardy property of means, we recall several notions that have been partly introduced and used 
in the paper \cite{PalPas16}. 

Let $I\subseteq\R$ be an interval and let $\M \colon\bigcup_{n=1}^{\infty} I^n \to I$ be an arbitrary mean. We say that $\M$ is \emph{symmetric}, \emph{(strictly) increasing}, and \emph{Jensen convex (concave)} if, for all $n\in\N$, the $n$ variable restriction $\M|_{I^n}$ is a symmetric, (strictly) increasing in each of its variables, and Jensen convex (concave) on $I^n$, respectively. It is worth mentioning that means are locally bounded functions, therefore, the so-called Bernstein--Doetsch theorem \cite{BerDoe15} implies that Jensen convexity (concavity) is equivalent to ordinary convexity (concavity). If $I=\R_+$, we can analogously define the notion of homogeneity of $\M$. 
Finally, the mean $\M$ is called 
\emph{repetition invariant} if, for all $n,m\in\N$ and $(x_1,\dots,x_n)\in I^n$, the following identity is satisfied
\Eq{*}{
  \M(\underbrace{x_1,\dots,x_1}_{m\text{-times}},\dots,\underbrace{x_n,\dots,x_n}_{m\text{-times}})
   =\M(x_1,\dots,x_n).
}

Having all these definitions, let us recall the two main theorems of the paper \cite{PalPas16}.
The first result provides a lower estimation of the Hardy constant.

\begin{thm}\label{prop:PalPas16}
Let $I \subset \R_+$ be an interval with $\inf I=0$ and $\M \colon \bigcup_{n=1}^{\infty} I^n \to I$ be a mean. Then,
for all non-summable sequences $(x_n)_{n=1}^{\infty}$ in $I$,
\Eq{*}{
\Hc\M \ge \liminf_{n \to \infty} x_n^{-1} \cdot \M\left(x_1,x_2,\ldots,x_n\right).
}
In particular,
\Eq{*}{
\Hc\M \ge \sup_{y\in I} \liminf_{n \to \infty} \frac ny \cdot \M\left(\frac y1,\frac y2,\ldots,\frac yn 
\right)=:\Est\M.
}
\end{thm}

Under stronger assumptions for the mean $\M$, the lower estimate obtained above becomes an equality by the following 
result.

\begin{thm}\label{thm:mainPaPa}
For every increasing, symmetric, repetition invariant, and Jensen 
concave mean $\M \colon \bigcup_{n=1}^{\infty} \R_+^n \to\R_+$ the equality $\Hc\M=\Est\M$ holds.

If, in addition, $\M$ is also homogeneous, then 
\Eq{*}{
  \Hc\M= \lim_{n \to \infty} n \cdot \M\left(1,\tfrac 12,\ldots,\tfrac 1n \right),
}
in particular, this limit exists.
\end{thm}

Upon taking $\M$ to be a power mean in the above theorem, the Hardy--Landau--Knopp--Carleman inequality \eq{H} can 
easily be deduced. For the details, see \cite{PalPas16}.


A deeper look into the paper \cite{PalPas16} shows that Theorem~\ref{thm:mainPaPa} could be split into two parts with an intermediate state of so-called Kedlaya mean. 

Mean $\M \colon \bigcup_{n=1}^\infty I^n \to I$ ($I$ is an interval) is a \emph{Kedlaya mean} if 
\Eq{DKI}{
\A\big(x_1,\M(x_1,x_2),\dots,\M(x_1,x_2,\dots,x_n)\big)
\le \M\big(x_1,\A(x_1,x_2),\dots,\A(x_1,x_2,\dots,x_n)\big) 
}
for every $n \in \N$ and $x \in I^n$. Here and throughout this paper $\A$ will stand for the standard (or weighted) arithmetic mean. 

The motivation for the above terminology came from the paper \cite{Ked94} by Kedlaya, where he proved that the 
geometric mean satisfies the inequality above, i.e., it is a Kedlaya mean. This result provided an affirmative answer to a conjecture by Holland \cite{Hol92}. A more general theorem has recently been established by the authors 

\begin{thm}[\cite{PalPas16}, Theorem 2.1] \label{thm:sufKedM}
Every symmetric, Jensen concave and repetition invariant mean is a Kedlaya mean.
\end{thm}
Moreover, in the proof of Theorem~\ref{thm:mainPaPa} the main tool was the following (nowhere explicitly formulated) statement

\begin{prop}\label{prop:KedtoHar}
For every monotone Kedlaya mean $\M\colon \bigcup_{n=1}^{\infty} I^n \to I$ (where $I$ is an interval with $\inf I=0$), the equality $\Hc\M=\Est\M$ holds.
\end{prop}

Obviously Theorem~\ref{thm:sufKedM} jointly with Proposition~\ref{prop:KedtoHar} imply the first part of Theorem~\ref{thm:mainPaPa}. To prove the second part, we need to show that the mentioned limit exists.

Recently an approach to weighted Kedlaya inequalities has been presented by the authors in \cite{PalPas17b}. In particular, a weighted counterpart of Theorem~\ref{thm:sufKedM} has been established (see notation of $\WQX$ and $\WRX$ below). It motivated us to look for a weighted analogue of Proposition~\ref{prop:KedtoHar}. The result obtained in this direction will be presented in Theorem~\ref{thm:HardyconstantforKedlaya}. 

\section{Weighted means}

At the moment weighted means will be introduced. First, let us underline that there is no broad agreement about the definition of weighted means. The one presented below was introduced in \cite{PalPas17b} in the process of reverse-engineering. The main idea was to cover most of the known weighted means (i.e. power, quasi-arithmetic, deviation, and quasi-deviation means) in the abstract setting. This consideration led us to introduce the following new definition.
 
\begin{defin}[Weighted means]
Let $I \subset \R$ be an arbitrary interval, $R \subset \R$ be a ring and, for 
$n\in\N$, define the set of $n$-dimensional weight vectors $W_n(R)$ by
\Eq{*}{
  W_n(R):=\{(\dotvec{\lambda_1}{\lambda_n})\in R^n\mid\lambda_1,\dots,\lambda_n\geq0,\,\lambda_1+\dots+\lambda_n>0\}.
}
\emph{A weighted mean on $I$ over $R$} or, in other words, \emph{an $R$-weighted mean on $I$} is a function 
\Eq{*}{
\M \colon \bigcup_{n=1}^{\infty} I^n \times W_n(R) \to I
}
satisfying the conditions (i)--(iv) presented below.
Elements belonging to $I$ will be called \emph{entries}; elements from $R$ -- \emph{weights}. 

\begin{enumerate}[(i)]
 \item \emph{Nullhomogeneity in the weights}: For all $n \in \N$, for all $(x,\lambda) \in I^n \times W_n(R)$, 
and $t \in R_+$,
 \Eq{*}{
   \M(x,\lambda)=\M(x,t \cdot \lambda),
 }
\item \emph{Reduction principle}: For all $n \in \N$ and for all $x \in I^n$, $\lambda,\mu \in W_n(R)$, 
\Eq{*}{
\M(x,\lambda+\mu)=\M(x\odot x,\lambda\odot\mu),
}
where $\odot$ is a \emph{shuffle operator}\footnote{This definition comes from the theory of computation. Perhaps the most famous (folk) result states that shuffling of two regular languages is again regular; see e.g. \cite{BloEsi97}.} defined as
\Eq{*}{
(\dotvec{p_1}{p_n})\odot (\dotvec{q_1}{q_n}):=(\dotvec{p_1,q_1}{p_n,q_n}).
}
\item \emph{Mean value property}: For all $n \in \N$, for all $(x,\lambda) \in I^n \times W_n(R)$
\Eq{*}{
\min(\dotvec{x_1}{x_n}) \le \M(x,\lambda)\le \max(\dotvec{x_1}{x_n}),
}
\item \emph{Elimination principle}: For all $n \in \N$, for all $(x,\lambda) \in I^n \times W_n(R)$ and for 
all $j\in\{1,\dots,n\}$ such that $\lambda_j =0$,
\Eq{*}{
\M(x,\lambda) = \M\big((x_i)_{i\in\{1,\dots,n\}\setminus\{j\}},(\lambda_i)_{i\in\{1,\dots,n\}\setminus\{j\}}\big),
}
i.e., entries with a zero weight can be omitted. 
\end{enumerate}
\end{defin}

From now on $I$ is an arbitrary interval, $R$ stands for an arbitrary subring of $\R$.
Let us introduce some natural properties of weighted means. A weighted mean $\M$ is said to be \emph{symmetric}, if for all $n \in \N$, $x \in I^n$, $\lambda \in W_n(R)$, and $\sigma \in S_n$, 
\Eq{*}{
\M(x,\lambda) =\M(x\circ\sigma,\lambda\circ\sigma).
}
We will call a weighted mean $\M$ \emph{Jensen concave} if, for all $n \in \N$, $x,y \in I^n$ and $\lambda \in W_n(R)$,
\Eq{E:JF2}{
\M \Big( \frac{x+y}2 , \lambda \Big) \ge\frac12 \big( \M(x,\lambda)+\M(y,\lambda) \big).
}

A weighted mean $\M$ is said to be \emph{continuous in the weights} if, for all $n \in \N$ and $x \in I^n$, the mapping $\lambda \mapsto \M(x,\lambda)$ is continuous on $W_n(R)$.
Finally, a weighted mean $\M$ is \emph{monotone} if, for all $n \in \N$, $x \in I^n$ and $\lambda \in W_n(R)$, the mapping $x_i \mapsto \M(x,\lambda)$ is increasing for all $i \in \{1,\dots,n\}$.

For the sake of convenience, we will use the sum-type abbreviation. If $\M$ is an $R$-weighted mean on $I$, $n\in\N$ and $(x,\lambda) \in I^n\times W_n(R)$, then we denote
\Eq{*}{
\Mm_{i=1}^n(x_i,\lambda_i):=\M\big((\dotvec{x_1}{x_n}),(\lambda_1,\dots,\lambda_n)\big).
}

Let us recall some basic properties of weighted means defined in this way.
First result binds nonweighted, repetition invariant means and $\Z$-weighted means.

\begin{thm}[\cite{PalPas17b},Theorem~2.3]\label{thm:weighted_nonweighted}
If $\M$ is a repetition invariant mean on $I$, then the formula  
\Eq{WeiQDef}{
\widetilde \M \big((\dotvec{x_1}{x_n}),(\lambda_1,\dots,\lambda_n)\big):=\M\big( \dotvec[\lambda_1]{x_1}{x_1},\dots,\dotvec[\lambda_n]{x_n}{x_n}\big)
}
defines a weighted mean on $I$ over $\Z$.

Conversely, if $\widetilde \M$ is a $\Z$-weighted 
mean on $I$, then 
\Eq{E:MMtilde}{
\M(\dotvec{x_1}{x_n}):=\widetilde \M\big((\dotvec{x_1}{x_n}),(\dotvec[n]11)\big)
}
is a repetition invariant mean on $I$. Furthermore these transformations are inverses of each other.
\end{thm}

Furthermore the following two easy statements were explicitly worded.

\begin{thm}\label{thm:sym}
If $\M$ is a symmetric repetition invariant mean on $I$, then the function $\widetilde\M$ defined by the formula \eq{WeiQDef} is a symmetric weighted mean on $I$ over $\Z$. 

Conversely, if $\widetilde \M$ is a symmetric $\Z$-weighted mean on $I$, then the function $\M$ defined by \eq{E:MMtilde} is a symmetric repetition invariant mean on $I$. 
\end{thm}

\begin{thm}\label{thm:conc}
If $\M$ is a Jensen concave repetition invariant mean on $I$, then the function $\widetilde\M$ defined by the formula \eq{WeiQDef} is a Jensen concave weighted mean on $I$ over $\Z$. 

Conversely, if $\widetilde \M$ is a Jensen concave $\Z$-weighted mean on $I$, then the function $\M$ defined by \eq{E:MMtilde} is a Jensen concave repetition invariant mean on $I$. 
\end{thm}

We will also need some extension theorem from paper \cite{PalPas17b}.

\begin{thm}
\label{thm:fieldex}
Let $I$ be an interval, $R \subset \R$ be a ring, $\M$ be a weighted mean defined on $I$ over $R$.
Then there exists a unique mean $\widetilde\M$ defined on $I$ over $R^*$ (which denotes the quotient field of $R$) such that 
\Eq{*}{
\widetilde \M \vert_{\bigcup_{n=1}^{\infty} I^n \times W_n(R)} =\M.
}
Moreover if $\M$ is symmetric/monotone then so is $\widetilde\M$.
\end{thm}

Having this we can extend means defined in Theorem~\ref{thm:weighted_nonweighted} to the field $\Q$.

Let us recall that, for $p\in\R$, the weighted power mean $\P_p\colon \bigcup_{n=1}^{\infty} \R_+^n \times W_n(\R) \to \R_+$ which is defined by
\Eq{*}{
\P_p (x,\lambda):= 
\begin{cases} 
\left(\frac{\lambda_1x_1^p+\lambda_2x_2^p+\cdots+\lambda_nx_n^p}{\lambda_1+\lambda_2+\cdots+\lambda_n} \right)^{1/p} &\quad \text{ if } p \ne 0, \\                                                          
\left(x_1^{\lambda_1}x_2^{\lambda_2} \cdots x_n^{\lambda_n} \right)^{1/(\lambda_1+\lambda_2+\cdots+\lambda_n)} &\quad \text{ if } p = 0,
\end{cases}
}
admits all properties (i)--(iv). 

In a more general setting, in the spirit of book \cite{HarLitPol34}, we can define weighted quasi-arithmetic means as follows.
If $I$ is an arbitrary interval and $f \colon I \to \R$ is continuous and monotone, then the weighted quasi-arithmetic mean $\QA{f} \colon \bigcup_{n=1}^{\infty} I^n \times W_n(\R) \to I$ is a function such that for all $n \in \N$ and a pair $x \in I^n$ with weights $\lambda \in W_n(\R)$,
\Eq{*}{
\QA{f}(x,\lambda):= f^{-1} \left( \frac{\lambda_1 f(x_1)+\lambda_2 f(x_2)+\cdots+\lambda_nf(x_n)}{\lambda_1+\lambda_2+\cdots+\lambda_n} \right).
}
This sequence of generalization could be continued to Bajraktarevi\'c means, to deviation (Dar\'oczy) means and to quasi-deviation means. Investigating these families however lies outside the scope of this paper and we just refer the reader to a series of papers by Losonczi \cite{Los70a,Los71a,Los71b,Los71c,Los73a,Los77} (for Bajraktarevi\'c means), Dar\'oczy \cite{Dar71b,Dar72b}, Dar\'oczy--Losonczi \cite{DarLos70}, Dar\'oczy--P\'ales \cite{DarPal82,DarPal83} (for deviation means) and by P\'ales \cite{Pal82a,Pal83b,Pal84a,Pal85a,Pal88a,Pal88d,Pal88e} (for deviation and quasi-deviation means).

\subsection{Weighted Kedlaya property}
Like in the paper \cite{PalPas17b}, we are going to introduce the notion of the weighted Kedlaya inequality. 

To have a weighted counterpart of the Kedlaya inequality, we have to take weight sequences $\lambda$ from $R$ with a positive first member.
Therefore, for a given ring $R$, we define
\Eq{*}{
  W^0(R)&:=\{\lambda\in R^\N\mid \lambda_1>0,\,\lambda_2,\lambda_3,\dots\geq0\}.
}

For a weight sequence $\lambda \in W^0(R)$, we say that a weighted mean $\M \colon  \bigcup_{n=1}^{\infty} I^n \times W_n(R) \to I$ satisfies the \emph{$\lambda$-weighted Kedlaya inequality}, or shortly, the \emph{$\lambda$-Kedlaya inequality} if 
\Eq{*}{
\Ar_{k=1}^n \left( \Mm_{i=1}^k (x_i,\lambda_i),\:\lambda_k\right) \le \Mm_{k=1}^n \left( \Ar_{i=1}^k (x_i,\lambda_i),\:\lambda_k\right)
  \qquad(n\in\N,\, x\in I^n).
}

In fact the nonincreasingness of the ratio sequence $\big(\tfrac{\lambda_i}{\lambda_1+\cdots+\lambda_i}\big)$ will be a key assumption for Kedlaya type inequalities, therefore, we also set
\Eq{*}{
   V(R)&:=\big\{\lambda\in W^0(R)\mid \big(\tfrac{\lambda_i}{\lambda_1+\cdots+\lambda_i}\big)_{i=1}^\infty \mbox{ is nonincreasing}\big\}.
}
In fact, in 1999 Kedlaya \cite{Ked99} proved that the geometric mean satisfies the $\lambda$-weighted Kedlaya inequality for all $\lambda \in V(\R)$. This result has been generalized recently by the authors \cite{PalPas17b} to the family of symmetric, Jensen concave means. More precisely, the following theorem has been established.
\begin{thm}
\label{thm:KedVQVR}
Every symmetric and Jensen concave $\Q$-weighted mean (resp.\ $\R$-weighted mean which is continuous in the weights) satisfies the $\lambda$-weighted Kedlaya inequality for all $\lambda \in V(\Q)$ (resp.\ $\lambda \in V(\R)$). 
\end{thm}

In fact, we will sometimes assume that a mean is a $\lambda$-Kedlaya mean and the above theorem delivers us a sufficient condition (compare Theorem~\ref{thm:HardyconstantforKedlaya} and related Corollaries~\ref{cor:HardyconstantVQ}, \ref{cor:HardyconstantVR}).

\subsection{Weighted Hardy property}

Similarly as in \cite{PalPas16}, Kedlaya inequality lead us to the Hardy property (with an optimal constant). Nevertheless, to make advantage of weighted Kedlaya inequality in struggling with the Hardy property, we need to define its weighted counterpart. Such a definition is a natural extension of the non-weighted setup.


\begin{defin}[Weighted Hardy property]
Let $I$ be an interval with $\inf I=0$, $R \subset \R$ be a ring. For an $R$-weighted mean $\M$ on $I$ and weights $\lambda \in W^0(R)$, let $C$ be the smallest extended real number such that, for all sequences $(x_n)$ in $I$,
\Eq{*}{
\sum_{n=1}^{\infty} \lambda_n \cdot \Mm_{i=1}^n\big(x_i,\lambda_i\big) \le C \cdot \sum_{n=1}^{\infty} \lambda_nx_n.
}
We call $C$ the \emph{$\lambda$-weighted Hardy constant of $\M$} or the \emph{$\lambda$-Hardy constant of $\M$} and denote it by $\Hc[\lambda]\M$. Whenever this constant is finite, then $\M$ is called a \emph{$\lambda$-weighted Hardy mean} or simply a \emph{$\lambda$-Hardy mean}.
\end{defin}

This definition is an extension of the Hardy constant (and, consequently, the Hardy property). Indeed, if we define $\vone:=(1,1,1,\dots)$ then, by Theorem~\ref{thm:weighted_nonweighted}, the weighted mean $\widetilde\M$ with weights $\vone$ could be associated with the non-weighted mean $\M$, and (in the setting of this theorem) the following equality is valid
\Eq{*}{
\Hc[\vone]{\widetilde\M} = \Hc{\M}.
}
There appears a natural question. What is a relation between being $\lambda$-Hardy and $\vone$-Hardy. Luckily, we have a simple (in its wording) property which generalizes the result of Elliott--Copson \eq{E:EllCop} (see \cite{Cop27} and \cite{Ell26}).

\begin{thm}
\label{thm:mu1}
For every symmetric and monotone mean $\M$ on $I$ over $R$, we have
\Eq{*}{
\Hc[\vone]\M=\sup_{\lambda \in W^0(R)} \Hc[\lambda]\M.
}
\end{thm}

Its technical and quite long proof is shifted upon the last section. As an immediate consequence we obtain
\begin{cor}
Let $\widetilde \M$ be a symmetric and monotone  on $I$ over $R$. Then the following conditions are equivalent:
\begin{enumerate}[(i)]
\item $\widetilde \M$ is a $\lambda$-Hardy mean for all $\lambda \in W^0(R)$;
\item $\widetilde \M$ is a $\vone$-Hardy mean;
\item $\M$ defined by \eq{E:MMtilde} is a Hardy mean.
\end{enumerate}
\end{cor}



\section{Auxiliary Results}

In this section we prove a number of technical lemmas which will be useful in the forthcoming sections.

First, a purely analytic fact is established. This is followed by results that are directly related to weighted Hardy property. Throughout this section, let $\lambda \in W^0(\R)$ and set $\Lambda_n:=\lambda_1+\cdots+\lambda_n$ for $n\in\N$.

\begin{lem}
\label{lem:2}
The series $\sum\lambda_n$ and $\sum\lambda_n/\Lambda_n$ are equi-convergent (either both convergent or both divergent).
\end{lem}

Indeed, if $\sum_{n=1}^{\infty} \lambda_n<\infty$, then $\sum_{n=1}^{\infty} \lambda_n/\Lambda_n \le \sum_{n=1}^{\infty} \lambda_n/\Lambda_1<\infty$. 
The reversed implication is due to Abel, see \cite[p.\ 125]{Kno56}.

Now we turn to results directly related to means. The first two statements are about properties of the Hardy constant, while the last one is a sort of rearranging property of a weighted mean in a case of a nonincreasing function.

The following lemma is somehow related to the so-called Hardy sequence (cf. \cite[Prop. 3.1]{PalPas16}).
\begin{lem} 
\label{lem:5}
Let $\M$ be an $R$-weighted mean on $I$. Then, for all $n\in \N$ and $x \in I^n$,
\Eq{FinWH}{
\sum_{i=1}^n \lambda_i \cdot \Mm_{j=1}^i \big(x_j,\lambda_j\big)\le \Hc[\lambda]\M \sum_{i=1}^n \lambda_ix_i.
}
\end{lem}

\begin{proof}
Take $\varepsilon \in I$ and $x_m:=\min \big(\varepsilon/(\lambda_m 2^m),\varepsilon \big)$ for $m>n$. Then we have
\Eq{*}{
\sum_{i=1}^n \lambda_i \cdot \Mm_{j=1}^i \big(x_j,\lambda_j\big) 
&\le \sum_{i=1}^\infty \lambda_i \cdot \Mm_{j=1}^i \big(x_j,\lambda_j\big) 
\le \Hc[\lambda]\M \sum_{i=1}^\infty \lambda_ix_i\\
&\le \Hc[\lambda]\M \left(\sum_{i=1}^n \lambda_ix_i+\sum_{i=n+1}^\infty \frac{\varepsilon}{2^i}\right) 
\le \Hc[\lambda]\M \left(\varepsilon+\sum_{i=1}^n \lambda_ix_i\right).
}
Now we can pass the limit $\varepsilon \to 0$ to obtain \eq{FinWH}.
\end{proof}

Having this already proved, we can present a weighted analogue of \cite[Thm 3.3]{PalPas16}. By the virtue of Stolz's theorem \cite{Sto85}, its proof is significantly shortened.

\begin{lem} 
\label{lem:4}
Let $\M$ be an $R$-weighted mean on $I$. If $\sum_{n=1}^\infty \lambda_nx_n=\infty$ then 
\Eq{*}{
\Hc[\lambda]\M \ge \liminf_{n \to \infty} \frac1{x_n} \Mm_{i=1}^n \big(x_i,\lambda_i\big).
}
\end{lem}

\begin{proof}
By Stolz's theorem and Lemma~\ref{lem:5}, we have
\Eq{*}{
\Hc[\lambda]\M 
&\ge \liminf_{N \to \infty,\,\lambda_N>0} \frac{\sum_{n=1}^N \lambda_n \cdot \Mm_{i=1}^n \big(x_i,\lambda_i\big)}{\sum_{n=1}^N \lambda_nx_n} \\
&\ge \liminf_{n \to \infty,\,\lambda_n>0} \frac{\lambda_n\Mm_{i=1}^n \big(x_i,\lambda_i\big)}{\lambda_nx_n}
\ge \liminf_{n \to \infty} \frac{\Mm_{i=1}^n \big(x_i,\lambda_i\big)}{x_n},
}
which was to be shown.
\end{proof}

\section{Main results}

In this section we will prove an important relation between the $\lambda$-Kedlaya and the $\lambda$-Hardy property. Having this, we will use the notation of $\WRX$ and $\WQX$ to present a handy characterization of the $\lambda$-Hardy property. In fact, a lot of statements will depend on the summability of the weight sequence $(\lambda_n)$. 

\begin{thm} 
\label{thm:HardyconstantforKedlaya}
Let $\M$ be an $R$-weighted mean on $I$ and $\lambda \in W^0(R)$. 
Define
\Eq{*}{
\Est[\lambda]\M:=\sup_{y>0} \liminf_{n \to \infty} \frac {\lambda_1+\lambda_2+\cdots+\lambda_n}y \cdot \Mm_{k=1}^n \Big(\frac{y}{\lambda_1+\lambda_2+\cdots+\lambda_k},\lambda_k\Big).
}
\begin{enumerate}[(i)]
\item If $\sum_{n=1}^\infty\lambda_n=\infty$, then $\Hc[\lambda]\M \ge \Est[\lambda]\M$.
\item If $\M$ is monotone and satisfies the $\lambda$-Kedlaya inequality, then $\Hc[\lambda]\M \le \Est[\lambda]\M$.
\end{enumerate}
\end{thm}


\begin{proof}
Denote the partial sum of $\lambda_1+\cdots+\lambda_k$ by $\Lambda_k$. In the first part, Lemma~\ref{lem:2} implies
\Eq{*}{
\sum_{n=1}^{\infty} \lambda_n \cdot \frac{y}{\Lambda_n}=\infty \qquad \text{ for all }\qquad y>0.
}
Consequently, by Lemma~\ref{lem:4}, 
\Eq{*}{
\Hc[\lambda]\M \ge \liminf_{n \to \infty} \frac{\Lambda_n}y \Mm_{k=1}^n \Big( \frac{y}{\Lambda_k},\lambda_k\Big) \qquad \text{ for all }\qquad y>0.
}
Finally, we can take the supremum over all positive $y$ and obtain $\Hc[\lambda]\M \ge \Est[\lambda]\M$.

To prove part (ii), let $x \in \ell^1(\lambda)$ be a sequence of positive numbers and $y_0:=\sum_{n=1}^{\infty} \lambda_nx_n$. Then
\Eq{*}{
m_k=\frac1{\Lambda_k}\sum_{i=1}^k \lambda_i \cdot x_i \le \frac{y_0}{\Lambda_k}, \qquad k \in \N. 
}
The $(n,\lambda)$-Kedlaya inequality applied to the vector $(x_1,x_2,\dots,x_n)$ and the monotonicity of $\M$ imply
\Eq{*}{
\sum_{k=1}^n \lambda_k \cdot \Mm_{i=1}^k (x_i,\lambda_i) 
\le \Lambda_n \cdot \Mm_{k=1}^n \Big( m_k,\lambda_k\Big) 
\le \Lambda_n \cdot \Mm_{k=1}^n \Big( \frac{y_0}{\Lambda_k},\lambda_k\Big).
}
Upon taking the $\liminf$ as $n$ tends to $\infty$, we obtain
\Eq{*}{
\sum_{k=1}^\infty \lambda_k \cdot \Mm_{i=1}^k (x_i,\lambda_i) 
\le \bigg( \liminf_{n \to \infty} \frac{\Lambda_n}{y_0} \cdot \Mm_{k=1}^n \Big( \frac{y_0}{\Lambda_k},\lambda_k\Big) \bigg) \cdot y_0
\le \Est[\lambda]\M \sum_{n=1}^{\infty} \lambda_nx_n.
}
Therefore, the desired inequality $\Hc[\lambda]\M \le \Est[\lambda]\M$ follows.
\end{proof}

At the moment, using Theorem~\ref{thm:KedVQVR}, we obtain two direct consequences of Theorem~\ref{thm:HardyconstantforKedlaya}.

\begin{cor}
\label{cor:HardyconstantVQ}
Let $\M$ be a symmetric, monotone and Jensen-concave $\Q$-weighted mean and $\lambda \in \WQX$. Then $\Hc[\lambda]\M \le \Est[\lambda]\M$. Furthermore, if $\sum_{n=1}^\infty\lambda_n=\infty$, then $\Hc[\lambda]\M = \Est[\lambda]\M$.
\end{cor}

\begin{cor}
\label{cor:HardyconstantVR}
Let $\M$ be a symmetric, monotone and Jensen-concave $\R$-weighted mean which is continuous in the weights and $\lambda \in \WRX$. Then $\Hc[\lambda]\M \le \Est[\lambda]\M$. Furthermore, if $\sum_{n=1}^\infty\lambda_n=\infty$, then $\Hc[\lambda]\M = \Est[\lambda]\M$.
\end{cor}

We can also apply this theorem to justify $\lambda$-Hardy property.
\begin{cor} 
Let $\M$ be an $R$-weighted mean on $I$ and $\lambda \in W^0(R)$. 
\begin{enumerate}[(i)]
\item If $\sum_{n=1}^\infty\lambda_n=\infty$ and $\Est[\lambda]\M=\infty$, then $\M$ is not a $\lambda$-Hardy mean.
\item If $\M$ is a monotone mean which satisfies the $\lambda$-Kedlaya inequality and $\Est[\lambda]\M$ is finite, then $\M$ is a $\lambda$-Hardy mean.
\end{enumerate}
\end{cor}

\section{Proof of Theorem~\ref{thm:mu1}}

Let us mention some further definitions and notations from \cite{PalPas17b}. Instead of explicitly writing down weights, we can consider a function with finite range as the argument of the given mean. Let $R$ be a subring of $\R$. We will denote its quotient field (the smallest field generated by $R$) by $R^*$. We say that $D\subseteq\R$ is an \emph{$R$-interval} if $D$ is of the form $[a,b)$, where $a,b\in R$.

Given an $R$-interval $D=[a,b)$, a function $f\colon D\to I$ is called \emph{$R$-simple} if there exist $n \in \N$ and a partition of $D$ into $R$-intervals $\{D_i\}_{i=1}^n$ such that $\sup D_i=\inf D_{i+1}$ for $i\in\{1,\dots,n-1\}$ and $f$ is constant on each subinterval $D_i$. Then, for an $R$-weighted mean $\M$ on $I$, we define
\Eq{*}{
\Mm_a^b f(x)dx:=\Mm_{i=1}^n (f|_{D_i},|D_i|)=\M((f|_{D_1},\dots,f|_{D_n}),(|D_1|,\dots,|D_n|)).
}

In this setting, $\M$ is symmetric if and only if, for every pair of $R$-simple functions $f,g \colon D \to I$ which have the same
distribution, the equality $\M f(x)dx=\M g(x)dx$ holds. Similarly, $\M$ is monotone if and only if, for every pair of $R$-simple functions $f,g \colon D \to I$ with $f\le g$, the inequality $\M f(x)dx \le \M g(x)dx$ is valid.

Furthermore, for an $R$-interval $[p,\,q) \subset D$ and function $f$ like above, we will keep all integral-type convections. For instance,
\Eq{*}{
\Mm_{[p,\,q)}f(x)dx=\Mm_p^q f(x)dx.
}

Let us now present some simple result related to decreasing functions.

\begin{lem}
\label{lem:decr}
Let $\M$ be a $R^*$-weighted, monotone mean on $I$ and $a \in R^* \cap (0,\infty)$. Then, for any nonincreasing $R^*$-simple function $f\colon [0,a) \to I$, the mapping $F\colon R^*\cap(0,a] \to I$ given by $F(u):=\Mm_0^uf(t)\:dt$, is nonincreasing.
\end{lem}
\begin{proof}
Fix $p,\,q \in R^*\cap(0,a]$ with $q<p$. As $f$ is decreasing, we know that $f(\tfrac pq t) \le f(t)$ for all $t \in [0,q)$. Therefore, by nullhomogeneity in the weights and monotonicity,
\Eq{*}{
F(p)=\Mm_0^p f(t) \:dt= \Mm_0^q f\Big(\dfrac pq \cdot t\Big)\:dt \le \Mm_0^q f(t)\:dt=F(q),
}
which was to be proved.
\end{proof}

By Theorem~\ref{thm:fieldex} we know that $\M$ has a unique extension to an $R^*$-weighted mean on $I$. As the weights are fixed (and belong to $R$) one can assume without loss of generality that we are dealing with weight sequence from $R^*$. Consequently, as it is handy, $\M$ is a $R^*$-weighted mean. 

To verify Theorem~\ref{thm:mu1}, it suffices to prove that, for all $N\in \N$, $\lambda \in W^0(R^*)$ and $x \in I^N$, there holds
\Eq{mu1F}{
\sum_{n=1}^N \lambda_n \Mm_{i=1}^n \big(x_i,\lambda_i\big) \le \Hc[\vone] \M \sum_{n=1}^N \lambda_nx_n.
}
Indeed, if we pass the limit $N \to \infty$, this inequality would imply $\Hc[\lambda]\M \le\Hc[\vone]\M$.

This proof is split into two parts. In fact each part can be formulated as a separate lemma.
\begin{lem}
\label{lem:monotonereduction}
Let $\M$ be a monotone $R^*$-weighted mean on $I$. Then, for all $N \in \N$, for all nonincreasing sequences $x \in I^N$ and weights $\lambda\in W^0_N(R^*)$, the inequality \eq{mu1F} is valid.
\end{lem}

\begin{lem}
\label{lem:rearrangement}
Let $\M$ be a symmetric and monotone $R$-weighted mean on $I$. Then, for all $N\in\N$, for all vectors $x\in I^N$ and weights $\lambda \in W^0_N(R)$, there exist $M \in \N$, a nonincreasing sequence $y\in I^M$ and a weight sequence $\psi\in W^0_M(R)$ such that
\Eq{equidistributed}{
\sum_{\{n\colon x_n=t\}} \lambda_n =\sum_{\{m\colon y_m=t\}} \psi_m
}
for all $t\in\R$ and
\Eq{Nonincreq}{
\sum_{n=1}^N \lambda_n \Mm_{i=1}^n \big(x_i,\lambda_i\big)&\le \sum_{m=1}^M \psi_m \Mm_{i=1}^m \big(y_i,\psi_i\big).
}
\end{lem}

Let us underline that the the fact that sum of $\lambda$-s and $\psi$-s are equal is not used in the proof of main theorem, however it could be useful in potential another use.

Having these two lemmas, for a given sequence $x=(x_1,x_2,\dots)$ with weights $\lambda \in W^0(R^*)$ and $N\in\N$, we can apply Lemma~\ref{lem:rearrangement} and then Lemma~\ref{lem:monotonereduction} to a vector $y\in I^M$ with corresponding weights $\psi$ to obtain
\Eq{*}{
\sum_{n=1}^N \lambda_n \Mm_{i=1}^n \big(x_i,\lambda_i\big)&\le \sum_{m=1}^M \psi_m \Mm_{i=1}^m \big(y_i,\psi_i\big) \le \Hc[\vone]\M \sum_{n=1}^M \psi_ny_n = \Hc[\vone]\M \sum_{n=1}^N \lambda_nx_n 
}
Then, if we pass the limit $N\to\infty$, we get
\Eq{*}{
\sum_{n=1}^{\infty} \lambda_n \Mm_{i=1}^n \big(x_i,\lambda_i\big)\le \Hc[\vone]\M \sum_{n=1}^{\infty} \lambda_nx_n \quad (\lambda \in W^0(R^*)\:),
}
which obviously implies $\Hc[\lambda]\M\le\Hc[\vone]\M$. As $\vone \in W^0(R)$, the equality in Theorem~\ref{thm:mu1} simply follows.

In order to make the proofs more compact, define $\Lambda_n:=\lambda_1+\cdots+\lambda_n$ for $n \in\{1,\dots,N\}$. In view of the nullhomogeneity of the mean $\M$, we may also assume that $\Lambda_N=1$. Now, define the function $f \colon [0,1) \to \R$ as follows
\Eq{*}{
f|_{[\Lambda_{n-1},\Lambda_n)}:=x_n,\qquad n \in \{1,\dots,N\}.
}
Then, we have that
\Eq{*}{
  \Mm_0^{\Lambda_n} f(x)dx=\Mm_{i=1}^n \big(x_i,\lambda_i\big), \qquad n \in \{1,\dots,N\}.
}

\begin{proof}[Proof of Lemma~\ref{lem:monotonereduction}]
First observe that, if $\Hc[\vone]\M=\infty$, then this lemma is trivial. From now on suppose that $\Hc[\vone]\M$ is finite.
Define, for $j\in\N$, the function $f_j \colon [0,1) \to I$ by 
\Eq{*}{
f_j|_{ [n/j,\,(n+1)/j)}:=f\big(\tfrac{n}{j}\big) \quad \text{ for all } n\in\{0,\dots,j-1\}.
}
As the sequence $x$ is nonincreasing, thus $f$ is nonincreasing, too. Therefore, $f \le f_j$ and $f_j$ is nonincreasing for every $j \in \N$. Thus, by Lemma~\ref{lem:decr}, so is the function $C_j \colon [0,1) \to I$ given by
\Eq{*}{
C_j(t):=\begin{cases}
         \inf\limits_{\substack{s\le t \\ s \in R^*}}\Mm_0^s f_j(x) dx &\mbox{ if }t\in(0,1), \\[2mm]
         x_1 &\mbox{ if }t=0,
        \end{cases} \qquad (j\in\N).
}
As $C_j$ is monotonic, it is also Riemann integrable. Using these properties, we get
\Eq{*}{
\lambda_n \cdot \Mm_{i=1}^n \big(x_i,\lambda_i\big) &= \lambda_n \cdot \Mm_0^{\Lambda_n} f(x) dx \le \lambda_n \cdot \Mm_0^{\Lambda_n} f_j(x) dx\\
&= \lambda_n \cdot C_j(\Lambda_n)= \int_{\Lambda_{n-1}}^{\Lambda_n} C_j(\Lambda_n) dx\le \int_{\Lambda_{n-1}}^{\Lambda_n} C_j(x) dx.
}
Therefore, for all $j\in\N$,
\Eq{Fin1}{
\sum_{n=1}^{N} \lambda_n \cdot \Mm_{i=1}^n \big(x_i,\lambda_i\big) 
&\le \int_{0}^1 C_j(x) dx.
}

We are now going to majorize the right hand side of this inequality. Observe first that, for all $j \in\N$,
\Eq{SP1}{
\int_{0}^{\frac1{j}} C_j(x)dx \le \frac1j\cdot C_j(0) = \frac{x_1}j.
}
Furthermore, for all $j,\,n \in \N$ such that $n<j$,
\Eq{SP2}{
\int_{\frac{n}{j}}^{\frac{n+1}{j}} C_j(x)dx 
\le \frac1j \cdot C_j\big(\tfrac{n}j \big)
=\frac1j\cdot \Mm_0^{\frac{n}{j}} f_j(x) dx
=\frac1j\cdot \Mm_{i=0}^{n} \big(f_j \big(\tfrac{i}{j}\big),1\big).
}
If we now sum up \eq{SP1} and \eq{SP2} for all $n\in\{1,\dots,j-1\}$, we obtain, for all $j\ge 2$,
\Eq{Fin2}{
\int_{0}^1 C_j(x) dx 
\leq \frac1j\bigg( x_1 + \sum_{n=1}^{j-1} \Mm_{i=0}^{n} \big(f_j \big(\tfrac{i}{j}\big),1\big)\bigg).
}
However, $\M$ is a $\vone$-weighted Hardy mean. In this setting by \cite[Proposition 3.1]{PalPas16}, 
we have that finite estimation announced in the definition of Hardy constant remains valid for finite sequences too. That is
\Eq{Fin3}{
\sum_{n=1}^{j-1} \Mm_{i=0}^{n} \big(f_j \big(\tfrac{i}{j}\big),1\big)
\le \Hc[\vone]\M \cdot \sum_{n=0}^{j-1} f_j\big(\tfrac{n}{j}\big) \qquad(j\ge 2).
}
Moreover, as $f$ is nonincreasing, we have
\Eq{Fin4}{
\frac1j  \sum_{n=0}^{j-1} f_j\big(\tfrac{n}{j}\big)
=\frac1j  \sum_{n=0}^{j-1} f\big(\tfrac{n}{j}\big)
\le \frac{x_1}{j}+\int_0^1 f(x) dx 
= \frac{x_1}{j}+\sum_{n=1}^{N} \lambda_n x_n \qquad(j\ge 2).
}
Now combining \eq{Fin1}, \eq{Fin2}, \eq{Fin3}, and \eq{Fin4}, for $j \ge 2$, we obtain
\Eq{*}{
\sum_{n=1}^{N} \lambda_n \cdot & \Mm_{i=1}^n \big(x_i,\lambda_i\big) 
\le \int_{0}^1 C_j(x) dx 
\leq \frac1j \bigg( x_1 + \sum_{n=1}^{j-1} \Mm_{i=0}^{n} \big(f_j \big(\tfrac{i}{j}\big),1\big)\bigg)\\
&\le \frac1j \bigg(x_1+\Hc[\vone]\M \sum_{n=0}^{j-1} f_j\big(\tfrac{n}{j}\big)\bigg)
 \le \frac{(1+\Hc[\vone]\M)x_1}j+\Hc[\vone]\M \sum_{n=1}^{N} \lambda_n x_n.
}
Finally, as $j \to \infty$, we get \eq{mu1F}.
\end{proof}

Now we turn to the proof of Lemma~\ref{lem:rearrangement}. Let us stress that in this lemma the assumptions for the mean $\M$ are more restrictive. More precisely, we assume $\M$ to be not only monotone but also symmetric. On the other hand, we need $\M$ to be $R$-weighted instead of $R^*$-weighted only. However, in view of Theorem~\ref{thm:fieldex}, this difference is rather a technical one.

\begin{proof}[Proof of Lemma~\ref{lem:rearrangement}] 
Throughout this proof, let us denote by $g^*$ the right continuous nonincreasing rearrangement of an $R$-simple function $g\colon D\to\R$. It is easy to observe that $g^*$ is again $R$-simple. 

Without the loss of generality, we may assume that the members of the sequence $\lambda$ are positive.
Consider a strictly increasing sequence $(\Psi_m)_{m=0}^M \in R^{M+1}$ such that $\Psi_0=0$, $\Psi_M=\Lambda_N$, 
$(\Lambda_n)_{n=0}^N$ is a subsequence of $(\Psi_m)_{m=0}^M$, and $f^*$ is constant on all intervals $[\Psi_{m-1},\Psi_m)$, where $m\in\{1,2,\dots,M\}$.

Set $\psi_m:=\Psi_m-\Psi_{m-1}$ and $y_m$ to be the value of $f^*$ on $[\Psi_{m-1},\Psi_m)$; $m \in \{1,\dots,M\}$. 
Furthermore, for every $n \in \{0,\dots,N\}$ there exists a unique $i_n\in\{0,\dots,M\}$ such that $\Psi_{i_n}=\Lambda_n$. As $\Psi_M=\Lambda_N$, we obtain $i_N=M$; furthermore, by $\Lambda_0=0=\Psi_0$, we get $i_0=0$. 
 
Obviously $(y_m)$ is nonincreasing, $\sum_{n=1}^N \lambda_n=\Lambda_N=\Psi_M=\sum_{m=1}^M \psi_M$, and 
\Eq{*}{
 \sum_{n=1}^N \lambda_nx_n=\int_0^{\Lambda_N} f(x)\:dx=\int_0^{\Psi_M} f(x)\:dx= \int_0^{\Psi_M} f^*(x)\:dx=\sum_{m=1}^M \psi_my_m.
}
Therefore the only property, which remains to be proved is \eq{Nonincreq}.
One can easily see that
 \Eq{*}{
(f\vert_{[0,u)})^*(x) \le f^*(x), \qquad x \in [0,u),\:u \in R \cap [0,\Lambda_N).
} 
Thus, by the monotonicity of $\M$,
\Eq{*}{
\Mm_0^u(f\vert_{[0,u)})^*(x)\:dx \le \Mm_0^u f^*(x)\:dx,\quad u \in R \cap [0,\Lambda_N).
} 
But, by the definition,  $(f\vert_{[0,u)})^*$ and $f\vert_{[0,u)}$ have the same distribution. Whence, applying the symmetry of $\M$, we arrive at
\Eq{*}{
\Mm_0^u f(x)\:dx \le \Mm_0^u f^*(x)\:dx,\qquad u \in R \cap [0,\Lambda_N).
} 
Therefore, applying this inequality for $u=\Lambda_n$, we obtain
\Eq{eq:1st}{
  \sum_{n=1}^N \lambda_n \Mm_{i=1}^n \big(x_i,\lambda_i\big)
  =\sum_{n=1}^N\lambda_n\Mm_0^{\Lambda_n} f(x)\:dx 
  \le \sum_{n=1}^N\lambda_n\Mm_0^{\Lambda_n} f^*(x)\:dx.
}
Now, let us notice that
\Eq{*}{
 \lambda_n=\Lambda_n-\Lambda_{n-1}=\Psi_{i_n}-\Psi_{i_{n-1}}=\sum_{m=i_{n-1}+1}^{i_n} (\Psi_m-\Psi_{m-1}).
}
Therefore, by Lemma~\ref{lem:decr}, definition of $(i_n)$, and the identity above, we obtain
\Eq{*}{
\sum_{n=1}^N \lambda_n \cdot \Mm_0^{\Lambda_n} f^*(x) dx
&=\sum_{n=1}^N \Big(\sum_{m=i_{n-1}+1}^{i_n} (\Psi_m-\Psi_{m-1})\Big)  \Mm_0^{\Psi_{i_n}} f^*(x) dx\\
&\le \sum_{n=1}^N \sum_{m=i_{n-1}+1}^{i_n} (\Psi_m-\Psi_{m-1}) \Mm_0^{\Psi_m} f^*(x) dx \\
&=\sum_{m=1}^M (\Psi_m-\Psi_{m-1}) \Mm_0^{\Psi_m} f^*(x) dx=\sum_{m=1}^M \psi_m \cdot \Mm_{i=1}^m (y_i,\psi_i).
}
But this inequality combined with \eq{eq:1st} is exactly what \eq{Nonincreq} states. As this was the only remaining property to be verified, the proof is complete.
\end{proof}


\end{document}